\documentclass[a4paper, 12pt]{amsart}
\usepackage[T1]{fontenc}
\usepackage{amssymb,amscd}
\usepackage{amsthm}
\usepackage{amsmath}
\usepackage{amssymb}
\usepackage{eucal}

\newtheorem{theorem}{Theorem}[section]
\newtheorem{proposition}[theorem]{Proposition}
\newtheorem{lemma}[theorem]{Lemma}
\newtheorem{corollary}[theorem]{Corollary}

\theoremstyle{definition}
\newtheorem{definition}[theorem]{Definition}
\newtheorem{example}[theorem]{Example}
\newtheorem{remark}[theorem]{Remark}
\newtheorem{question}[theorem]{Question}

\numberwithin{equation}{section}

\begin{document}
\title[]
{Compact complement topologies and k-spaces}
\author{K. Keremedis, C. \"{O}zel, A. Pi\k{e}kosz, M. A. Al Shumrani, E. Wajch}
\address{Department of Mathematics,
University of Aegean,
Karlovassi, Samos 83200, Greece}
\email{kker@aegean.gr}
\address{Department of Mathematics, King Abdulaziz University,
P.O.Box: 80203 Jeddah 21589, Saudi Arabia}
\email{cenap.ozel@gmail.com}
\address{Institute of Mathematics, Cracow University of Technology, Warszawska 24, 31-155 Krak\'ow, Poland}
\email{pupiekos@cyfronet.pl}
\address{Department of Mathematics, King Abdulaziz University,
P.O.Box: 80203 Jeddah 21589, Saudi Arabia}
\email{maalshmrani1@kau.edu.sa}
\address{Siedlce University of Natural Sciences and Humanities, 3 Maja 54, 08-110 Siedlce, Poland}
\email{eliza.wajch@wp.pl}

\thanks{}

\subjclass{Primary: 54D50, 54D55, 54A35, 54E99. Secondary: 54D30,  54E35, 54E25}
\date{\today}

\begin{abstract}
Let $(X,\tau)$ be a Hausdorff space, where $X$ is an infinite set. The compact complement topology $\tau^{\star}$ on $X$ is defined by: $\tau^{\star}=\{\emptyset\} \cup \{X\setminus M, \text{where $M$ is compact in $(X,\tau)$}\}$. In this paper,  properties of the space $(X, \tau^{\star})$ are studied in $\mathbf{ZF}$ and applied to a characterization of $k$-spaces, to the Sorgenfrey line, to some statements independent of $\mathbf{ZF}$, as well as to partial topologies that are among Delfs-Knebusch generalized topologies. Between other results, it is proved that the axiom of countable multiple choice (\textbf{CMC}) is equivalent with each of the following two sentences: (i) every Hausdorff first countable space is a $k$-space, (ii) every metrizable space is a $k$-space. A \textbf{ZF}-example of a countable metrizable space whose compact complement topology is not first countable is given.

{\bf Keywords:} compact complement topology; countable multiple choice; $k$-space; sequential space;  Sorgenfrey line; Delfs-Knebusch generalized topological space; partial topology.

\end{abstract}

\maketitle

\section{Introduction}
The compact complement topology of the real line was considered, for instance, in Example 22 of the celebrated book by Steen and Seebach"Counterexamples in Topology" (\cite{SS}). We investigate this notion in a much wider context of Hausdorff spaces and of partially topological spaces that belong to the class of generalized topological spaces in the sense of Delfs-Knebusch (cf. \cite{DK} and \cite{Pie1}). Our results are proved in $\mathbf{ZF}$ if this is not otherwise stated. All axioms of $\mathbf{ZF}$ can be found in \cite{Ku}.

In section 2, we give elementary properties of the compact complement topology of a Hausdorff space. In particular, we show that if a Hausdorff space is locally compact and second countable, then its compact complement topology is second countable, while the compact complement topology of a non-locally compact metrizable space need not be first countable. We give an example of a countable metrizable space whose compact complement topology is not first countable. In section 3, a necessary and sufficient condition for a set to be compact with respect to the compact complement topology of a given Hausdorff space leads us to a new characterization of k-spaces. A well-known theorem of $\mathbf{ZFC}$ states that all first countable Hausdorff spaces are $k$-spaces (cf. Theorem 3.3.20 of \cite{E}). We show that this theorem may fail in $\mathbf{ZF}$. More precisely, we prove that, if $\mathcal{M}$ is a model of $\mathbf{ZF}$, then all Hausdorff first countable spaces of $\mathcal{M}$ are k-spaces if and only if all metrizable spaces in $\mathcal{M}$ are $k$-spaces which holds if and only if the axiom of countable multiple choice (Form 126 in \cite{HR}) is true in $\mathcal{M}$. In consequence, in some models of $\mathbf{ZF}$ there are metrizable spaces that are not $k$-spaces. We prove that if the Sorgenfrey line is a $k$-space, then the real line with its natural topology is sequential, so the Sorgenfrey line fails to be a $k$-space in some models of $\mathbf{ZF}$. In section 4, we introduce a notion of a compact complement partial topology corresponding to a given partial topology. Partially topological Delfs-Knebusch generalized topological spaces were introduced in Definition 2.2.67 of  \cite{Pie1}; however,  a more convenient than in \cite{Pie1} definition of a partial topology was given in \cite{OPSW}. 

\section{Basic properties of compact complement topologies}

Throughout this article, we assume that $\tau$ is a topology on an infinite set $X$ such that $(X,\tau)$ is a Hausdorff space.

\begin{definition} \label{star}
 We denote by $\mathcal{K}(\tau)$ the collection of all $\tau$-compact sets, i.e. of all sets that are compact in the space $(X, \tau)$. The \emph{compact complement topology} of $(X, \tau)$ is the collection
$$\tau^{\star}=\{\emptyset\} \cup \{X\setminus M: M\in\mathcal{K}(\tau)\}.$$
\end{definition}

Since it is true in $\mathbf{ZF}$ that a compact subspace of a Hausdorff space is closed (see Theorem 3.1.8 of \cite{E}), it is easy to check in $\mathbf{ZF}$ that $\tau^{\star}$ is a topology on $X$.

For a subset $Y$ of $X$ and a topology $\mathcal{T}$ on $X$, let
$$\mathcal{T}\vert Y=\{V\cap Y: V\in\mathcal{T}\}.$$
Then $(Y, \mathcal{T}\vert Y)$ is a topological subspace of $(X, \mathcal{T})$.

\begin{theorem}\label{theorem1} Let $Y\subseteq X$. The following conditions hold in $\mathbf{ZF}$:
\begin{enumerate}
\item[(i)] $\tau^{\star}\vert Y$ is coarser than $\tau\vert Y$, i.e., $\tau^{\star}\vert Y \subseteq \tau\vert Y$;
\item[(ii)] if $Y$ is compact in $(X, \tau)$, then $\tau^{\star}\vert Y=\tau\vert Y$;
\item[(iii)] $\tau^{\star}\vert Y=\tau\vert Y$ if and only if there exists a $\tau$-compact set $C$ such that $Y\subseteq C$.
\end{enumerate}
\end{theorem}

\begin{proof} To prove (i), it suffices to show that $\tau^{\star}\subseteq \tau$. Let $U \in \tau^{\star}$ and $U\neq\emptyset$. Then $U=X\setminus M$ for a compact subspace $M$ of $(X,\tau)$. Since a compact subspace of a Hausdorff space is closed, $M$ is closed in $(X,\tau)$. Hence, $U \in \tau$ and, in consequence, $\tau^{\star}\subseteq \tau$.

(ii) Suppose that $Y$ is $\tau$-compact and $V\in \tau$. Since $(X, \tau)$ is Hausdorff, the set $Y$ is $\tau$-closed, so $A=Y\cap (X\setminus V)$ is a $\tau$-closed subset of the $\tau$-compact set $Y$. Hence, $A$ is $\tau$-compact. Notice that $V\cap Y=Y\cap(X\setminus A)$. This implies that $V\cap Y\in\tau^{\star}\vert Y$ and $\tau\vert Y\subseteq\tau^{\star}\vert Y$.

(iii) If $C$ is a $\tau$-compact set such that $Y\subseteq C$, then since it follows from (ii) that $\tau\vert C=\tau^{\star}\vert C$, we immediately deduce that $\tau\vert Y=\tau^{\star}\vert Y$. Finally,  suppose that $Y$ is a subset of $X$ such that $\tau^{\star}\vert Y=\tau\vert Y$. Let $V\in\tau$ and $\emptyset\neq V\cap Y\neq Y$. Since $V\cap Y\in\tau^{\star}\vert Y$, there exists a $\tau$-compact set $K_0$ such that $V\cap Y=Y\setminus K_0$. Fix $x_0\in V\cap Y$. Then $x_0\notin K_0$. By the $\tau$-compactness of $K_0$, there exists a pair $U_1, U_2$ of disjoint members of $\tau$ such that $x_0\in U_1$ and $K_0\subseteq U_2$. Of course, $U_2\cap Y\neq\emptyset$ because $V\cap Y\neq Y$. Since $U_1\cap Y$ and $U_2\cap Y$ are both in $\tau^{\star}\vert Y$, there exist $\tau$-compact sets $K_1, K_2$ such that $U_i\cap Y=Y\setminus K_i$ for $i\in\{1,2\}$. Let $K=K_1\cup K_2$. Then $K$ is $\tau$-compact and $Y=Y\setminus (U_1\cap U_2)=(Y\setminus U_1)\cup (Y\setminus U_2)\subseteq K_1\cup K_2=K$.
\end{proof}

\begin{corollary} $(X, \tau)$ is compact if and only if $\tau=\tau^{\star}$.
\end{corollary}

\begin{remark}  \label{2.4}
In general, $\tau^{\star}\vert Y$ is not equal to $(\tau \vert Y)^{\star}$. For instance, if $Y$ is the open interval $(0, 1)$ of $\mathbb{R}$, while $\tau_{nat}$ is the usual topology of $\mathbb{R}$, then $\tau_{nat}^{\star}\vert Y\neq(\tau_{nat} \vert Y)^{\star}$ . If this does not lead to misunderstanding, we shall denote the space $(\mathbb{R}, \tau_{nat})$ by $\mathbb{R}$ and call it the real line.
\end{remark}

From the fact that $\tau^{\star}\subseteq \tau$, we have the following obvious results:

\begin{proposition} (i) If $(X,\tau)$ is separable,  so is $(X, \tau^{\star})$.\\

(ii) If $(X, \tau)$ is hereditarily separable, so is $(X, \tau^{\star})$.\\

(iii)  If $(X, \tau)$ is hereditarily Lindel\"of, so is $(X, \tau^{\star})$.\\

(iv) If $(X,\tau)$ is connected, so is $(X, \tau^{\star})$.
\end{proposition}

The statement that every infinite set is Dedekind infinite (Form 9 in \cite{HR}) is denoted by \textbf{Fin} in Definition 2.13 of \cite{Her}.

\begin{theorem} The following sentences are equivalent in $\mathbf{ZF}$:
\begin{enumerate}
\item[(i)] \textbf{Fin}.
\item[(ii)] For every discrete space $(X, \tau)$, the space $(X, \tau^{\star})$ is hereditarily separable.
\item[(iii)] For every uncountable discrete space $(X, \tau)$, the space $(X, \tau^{\star})$ is separable.
\end{enumerate}
\end{theorem}
\begin{proof} Let $(X, \tau)$ be a discrete space, i.e. $\tau=\mathcal{P}(X)$. If $X$ is countable, then, of course, $(X, \tau^{\star})$ is hereditarily separable. Consider the case when  $Y\subseteq X$ and $Y$ is uncountable. If \textbf{Fin} holds, then $Y$ is Dedekind infinite, so $Y$ contains an infinitely countable subset. It is clear that when $D$ is an infinitely countable subset of $Y$, then $D$ is dense in $(Y, \tau^{\star}\vert Y)$. Hence, (i) implies (ii). It is obvious that (ii) implies (iii) and that (iii) implies (i).
\end{proof}

\begin{corollary} It is consistent with $\mathbf{ZF}$ that there exists an uncountable discrete space $(X, \tau)$ such that $(X, \tau^{\star})$ is separable and not hereditarily separable.
\end{corollary}
\begin{proof} Let $\mathcal{M}$ be any model of $\mathbf{ZF}$ in which $\mathbb{R}$ contains an uncountable Dedekind finite set. For instance, $\mathcal{M}$ can be Cohen's original model $\mathcal{M}$1 of \cite{HR}. Then if $\tau$ is the discrete topology on $\mathbb{R}$, the space $(\mathbb{R}, \tau^{\star})$ is not hereditarily separable because for each uncountable Dedekind finite subset $Y$ of $\mathbb{R}$, the space $(Y, \tau^{\star}\vert Y)$ is not separable. Of course, $(\mathbb{R}, \tau^{\star})$ is separable because $\mathbb{R}$ is Dedekind infinite.
\end{proof}

\begin{proposition}
For every Hausdorff space $(X, \tau)$, the space $(X, \tau^{\star})$ is $T_{1}$.
\end{proposition}

\begin{proof}
Let $x \in X$.  Since finite sets are compact, we have that $X\setminus \{x\}$ is open in $(X, \tau^{\star})$. Hence, $(X, \tau^{\star})$ is a $T_{1}$-space.
\end{proof}

\begin{proposition}
$(X,\tau)$ is not compact if and only if $(X, \tau^{\star})$ is not Hausdorff. Moreover, if $(X,\tau)$ is not compact, then any two non-empty $\tau^{\star}$-open sets have a non-empty intersection.
\end{proposition}

\begin{proof}
Assume that $(X,\tau)$ is not compact. Let $U$ and $V$ be any two non-empty open sets in $(X, \tau^{\star})$. Then $X\setminus U$ and $X\setminus V$ are compact  in $(X,\tau)$, so $(X\setminus U)\cup (X\setminus V)$ is compact in $(X, \tau)$. Hence, $X\neq (X\setminus U)\cup (X\setminus V)= X\setminus (U\cap V)$. This implies that $U\cap V\neq\emptyset$; thus, $(X, \tau^{\star})$ is not Hausdorff. On the other hand, if we assume that $(X, \tau^{\star})$ is not Hausdorff, then, since $(X, \tau)$ is Hausdorff, we have $\tau\neq\tau^{\star}$, so $(X, \tau)$ is not compact by Corollary 2.3.
\end{proof}

\begin{corollary}
If $(X,\tau)$ is not compact, then the following conditions are satisfied:
\begin{enumerate}
\item[(i)]  every set $V\in\tau^{\star}$ is connected in $(X, \tau^{\star})$;
\item[(ii)]  $(X, \tau^{\star})$  is connected and locally connected.
\end{enumerate}
\end{corollary}

\begin{proof} Suppose that $(X, \tau)$ is not compact and that $\emptyset\neq V\in\tau^{\star}$. If $V$ were disconnected in $(X, \tau^{\star})$, there would exist a pair $U, W$ of non-empty disjoint members of $\tau^{\star}$ which would contradict Proposition 2.9. Hence, $V$ is connected in $(X, \tau^{\star})$. This is why (i) holds. Of course, (ii) follows from (i).
\end{proof}

\begin{remark} Some authors call a topological space \emph{hyperconnected} or \emph{irreducible} if all open sets of this space are connected. In the light of Corollary 2.10, if $(X, \tau)$ is not compact, the space $(X, \tau^{\star})$ is hyperconnected.
\end{remark}

\begin{theorem} If $(X, \tau)$ is locally compact and second countable, then $(X, \tau^{\star})$ is second countable.
\end{theorem}

\begin{proof} Assume that $\mathcal{B}$ is a countable open base of a locally compact Hausdorff space $(X, \tau)$. Let $\mathcal{A}$ be the collection of all sets $U\in\mathcal{B}$ which have compact closures $\text{cl}_{\tau}U$ in $(X, \tau)$. By the local compactness of $(X, \tau)$, the collection $\mathcal{A}$ is an open base of $(X, \tau)$. Let $[\mathcal{A}]^{<\omega}$ be the collection of all finite subcollections of $\mathcal{A}$.  We put
$$\mathcal{B}^{\star}=\{ X\setminus\text{cl}_{\tau}(\bigcup\mathcal{C}): \mathcal{C}\in [\mathcal{A}]^{<\omega}\}.$$
Then $\mathcal{B}^{\star}$ is a countable subcollection of $\tau^{\star}$. To check that $\mathcal{B}^{\star}$ is an open base of $(X, \tau^{\star})$, let us consider any non-empty set $V\in\tau^{\star}$ and a point $x\in V$. Let $K=X\setminus V$ and let $\mathcal{U}$ be the collection of all $U\in\mathcal{A}$ such that $x\notin\text{cl}_{\tau} U$. Since $(X, \tau)$ is Hausdorff, we have $K\subseteq\bigcup\mathcal{U}$. By the compactness of $K$, there exists a finite $\mathcal{U}_K\subseteq\mathcal{U}$ such that $K\subseteq\bigcup\mathcal{U}_K$. Let $W=X\setminus\text{cl}_{\tau}(\bigcup\mathcal{U}_K)$. Then $W\in\mathcal{B}^{\star}, x\in W$ and $W\subseteq V$.
\end{proof}

The axiom of countable choice, denoted by $\mathbf{CC}$ in \cite{Her}, states that every non-empty countable collection of non-empty sets has a choice function (see Form 8 in \cite{HR}). The axiom of countable choice for $\mathbb{R}$, denoted by $\mathbf{CC}(\mathbb{R})$ in \cite{Her}, states that every non-empty countable collection of non-empty subsets of $\mathbb{R}$ has a choice function (see Form 94 in \cite{HR}).

\begin{remark} In view of Exercise E3 to Section 4.6 of \cite{Her},  $\mathbf{CC}(\mathbb{R})$ is equivalent to the statement: for every second countable topological space $Z$,  every open base of $Z$ contains a countable open base of $Z$. Let  us notice that, if $\mathcal{M}$ is a model of $\mathbf{ZF}$ in which there exists a dense infinite Dedekind finite subset $D$ of $\mathbb{R}$, then it holds true in $\mathcal{M}$ that the collection $\mathcal{B}^{\star}$ of all sets of the form $\mathbb{R}\setminus\bigcup_{i\in n+1}[a_i, b_i]$ with $n\in\omega, a_i,b_i\in D$ and $a_i<b_i$ for each $i\in n+1$ is an open base of $(\mathbb{R}, \tau_{nat}^{\star})$ which does not contain a countable open base of $(\mathbb{R}, \tau_{nat}^{\star})$.
\end{remark}

We recall that a topological space $(Z, \mathcal{T})$ is \emph{submetrizable} if there exists a metrizable topology $\mathcal{T}'$ on $Z$ such that $\mathcal{T}'\subseteq\mathcal{T}$.

\begin{theorem} The following conditions are equivalent:
\begin{enumerate}
\item[(i)] $(X, \tau^{\star})$ is metrizable;
\item[(ii)] $(X, \tau^{\star})$ is submetrizable;
\item[(iii)] $(X, \tau)$ is a compact metrizable space.
\end{enumerate}
\end{theorem}

\begin{proof} Of course, (i) implies (ii). Assume (ii). If $(X, \tau)$ is not compact, then $(X, \tau^{\star})$ is not Hausdorff by Proposition 2.9. Hence, (ii) implies that $(X, \tau)$ is compact. In this case, $\tau=\tau^{\star}$ by Corollary 2.3. In consequence, $(X, \tau)$ is both compact and submetrizable. Since every compact submetrizable space is metrizable, (ii) implies (iii). That (iii) implies (i) follows from Corollary 2.3.
\end{proof}

\begin{proposition} Let $x_0\in X$. Then $\{x_0\}$ is of type $G_{\delta}$ in $(X, \tau^{\star})$ if and only if $X\setminus\{x_0\}$ is a $\sigma$-compact subspace of $(X, \tau)$.
\end{proposition}
\begin{proof} \emph{Necessity.} Suppose that $\{U_n:n\in\omega\}\subseteq \tau^{\star}$ and $\{x_0\}=\bigcap_{n\in\omega} U_n$. Then the sets $K_n= X\setminus U_n$ are all compact in $(X, \tau)$ and $X\setminus \{x_0\}=\bigcup_{n\in\omega} K_n$, so $X\setminus\{x_0\}$ is $\sigma$-compact in $(X, \tau)$.

\emph{Sufficiency.} Suppose that $X\setminus\{x_0\}=\bigcup_{n\in\omega} C_n$ where all the sets $C_n$ are compact in $(X, \tau)$. Then the sets $V_n= X\setminus C_n$ are all open in $(X, \tau^{\star})$ and $\{x_0\}=\bigcap_{n\in\omega} V_n$.
\end{proof}

\begin{corollary} If $(X, \tau)$ is not $\sigma$-compact, then the following conditions are satisfied:
\begin{enumerate}
\item[(i)] there does not exist a one-point set of type $G_{\delta}$ in $(X, \tau^{\star})$;
\item[(ii)] $(X, \tau^{\star})$ is not first countable;
\item[(iii)] $(X, \tau^{\star})$ is not second countable;
\item[(iv)] $(X, \tau^{\star})$ is not quasi-metrizable.
\end{enumerate}
\end{corollary}

\begin{remark} We denote by $\mathbb{S}$ the Sorgenfrey line, i.e the topological space $(\mathbb{R}, \tau_{\mathbb{S}})$ where $\tau_{\mathbb{S}}$ is the topology on $\mathbb{R}$ which has as an open base the collection of all half-open intervals $[a, b)$ where $a, b\in\mathbb{R}$ and $a<b$.  The Sorgenfrey line is one of the most frequently used examples of a submetrizable, quasi-metrizable but not metrizable space, so we shall pay a special attention to it.

The countable union theorem (Form 31 in \cite{HR}, abbreviated to $\mathbf{CUT}$ in \cite{Her}) states that countable unions of countable sets are countable sets. Let $\mathbf{CUT}(\mathbb{R})$ be the statement: for every family $\{A_n: n\in\omega\}$ of countable subsets of $\mathbb{R}$, the union $\bigcup_{n\in\omega}A_n$ is countable (see Form 6 in \cite{HR})). It is easy to prove in $\mathbf{[ZF+CUT(\mathbb{R})]}$ that the Sorgenfrey line is not $\sigma$-compact by using the following simple argument: since all compact subsets of $\mathbb{S}$ are countable, if $\mathbb{S}$ were $\sigma$-compact, $\mathbb{R}$ would be a countable union of countable sets; however, $\mathbb{R}$  cannot be a countable union of countable sets because $\mathbb{R}$ is uncountable. This is not a proof in $\mathbf{ZF}$ that the Sorgenfrey line is not $\sigma$-compact because $\mathbf{CUT}(\mathbb{R})$ fails in some models of $\mathbf{ZF}$ (see Theorem 10.6 of \cite{J}).
\end{remark}

\begin{proposition} In every model of $\mathbf{ZF}$,  the Sorgenfrey line is not $\sigma$-compact.
\end{proposition}
\begin{proof} Consider any countable collection $\{K_n: n\in\omega\}$ of compact sets of the Sorgenfrey line. Then all  the sets $K_n$ are countable, closed in $\mathbb{R}$ and they do not have left accumulation points in $\mathbb{R}$. Therefore, each $K_n$ is nowhere dense in $\mathbb{R}$. Since $\mathbb{R}$ is a separable completely metrizable space, by Theorem 4.102 of \cite{Her}, the interior in $\mathbb{R}$ of the set $\bigcup_{n\in\omega} K_n$ is empty. Hence, $\mathbb{R}\neq\bigcup_{n\in\omega} K_n$.
\end{proof}

\begin{corollary} The compact complement topology of the Sorgenfrey line is not first countable.
\end{corollary}

\begin{corollary} The compact complement topology of the Sorgenfrey line is not quasi-metrizable.
\end{corollary}

\begin{proposition} The compact complement topology of the real line $\mathbb{R}$ is quasi-metrizable.
\end{proposition}

\begin{proof} For $x\in\mathbb{R}$, let $m(x)=\min\{n\in\omega: \vert x\vert<n\}$. For each $x\in\mathbb{R}$ and $n\in\omega$, we define a set $G(n, x)$ by putting:
$$G(n, x)=(x-\frac{1}{2^{n+1}}, x+\frac{1}{2^{n+1}})\cup (-\infty, -m(x)-n-2)\cup(m(x)+n+2, +\infty).$$
It is clear that, for each $x\in\mathbb{R}$,  the collection $\mathcal{B}(x)=\{ G(n, x): n\in\omega\}$ is a base of neighbourhoods of $x$ in $(\mathbb{R}, \tau_{nat}^{\star})$. One can check by a simple calculation that the following condition  satisfied: for all $x, y\in\mathbb{R}$ and $n\in\omega$, if $y\in G(n+1, x)$, then $G(n+1, y)\subseteq G(n, x)$. Let us notice that Theorem 10.2 of \cite{Gr} (Chapter 10 of \cite{KV}) holds true in $\mathbf{ZF}$, so we can infer from it that $(\mathbb{R}, \tau_{nat}^{\star})$ is quasi-metrizable in $\mathbf{ZF}$.
\end{proof}

We are going to give a simple $\mathbf{ZF}$-example of a countable  metrizable space whose compact complement topology is not first countable. We shall use the following lemma in this and in the third section:

\begin{lemma} Let us assume that  $\{A_n: n\in\omega\}$ is a collection of non-empty pairwise disjoint sets, $A=\bigcup_{n\in\omega}A_n$ and $Z=A\cup\{\infty\}$ where $\infty\notin A$. For $x, y\in Z$ let  $d(x,y)=d(y,x)$ and $d(x,x)=0$; for each pair $x,y$ of distinct points of $Z$, let $d(x,y)=\max\{\frac{1}{2^n}, \frac{1}{2^m}\}$ if $x\in A_n$ and $y\in A_m$; moreover, let $d(x, \infty)=\frac{1}{2^n}$ if $x\in A_n$. Then the function $d: Z\times Z\to\mathbb{R}$ is a metric on $Z$ such that $A$ is not closed in $(Z, \tau(d))$, while each $A_n$ is a clopen discrete subspace of $(Z, \tau(d))$ where $\tau(d)$ is the topology on $Z$ induced by $d$.
\end{lemma}
\begin{proof} Using the fact that  $\max\{a,b\}\leq\max\{a, c\}+\max\{c, b\}$ for all non-negative real numbers $a,b,c$, one can easily check that $d$ is a metric on $Z$. Since $\infty\in\text{cl}_{\tau(d)} A$, the set $A$ is not closed in $(Z,\tau(d))$. It is obvious that each $A_n$  is a clopen discrete subspace of $(Z, \tau(d))$.
\end{proof}

\begin{example} Let  $A_n=\{n\}\times \omega$ for each $n\in\omega$ and let $A=\bigcup_{n\in\omega}A_n$ Take a point $\infty\notin A$ and put $Z=A\cup\{\infty\}$. Consider the metric $d$ on $Z$ defined in Lemma 2.22. Suppose that the point $(0, 0)$ has a countable base $\{V_n: n\in\omega\}$ of open neighbourhoods in $(Z, \tau({d})^{\star})$ where $\tau(d)$ is as in Lemma 2.22. We may assume that $V_n\subseteq V_0$ for each $n\in\omega$.  The sets $Z\setminus V_n$ are all $\tau(d)$-compact, while the sets $A_n$ are not $\tau(d)$-compact because they are infinite discrete subspaces of $(Z,\tau(d))$. Hence, $A_n\cap V_n\neq\emptyset$ for each $n\in\omega$. For $n\in\omega$, let $a_n=\min\{m\in\omega: (n, m)\in A_n\cap V_n\}$. We define points $x_n\in A_n\cap V_n$ by putting $x_n=(n, a_n)$ for $n\in\omega$. Notice that the set $K=\{ x_n: n\in\omega\setminus\{0\}\}\cup\{\infty\}$ is $\tau(d)$-compact, while $(0,0)\notin K$. Then $V=Z\setminus K$ is an open neighbourhood of $(0, 0)$ in $(Z, \tau(d)^{\star})$. There must exist  $n\in\omega$ such that $V_n\subseteq V$. This is impossible because $V_n\subseteq V_0$ and $x_n\in V_n$ for each $n\in\omega\setminus\{0\}$. The contradiction obtained proves that  $(Z, \tau(d)^{\star})$ is not first countable. Obviously, $(Z, \tau(d))$ is $\sigma$-compact because $Z$ is countable. Of course, $(Z, \tau(d))$ is second countable as a separable metrizable space.  The point $\infty$ is not a point of local compactness of $(Z, \tau(d))$. This example shows that, in Theorem 2.12, the assumption of local compactness of $(X, \tau)$ cannot be replaced by the assumption that the set of points of non-local compactness of $(X,\tau)$ is finite. 
\end{example}

An arbitrary example of a metrizable second countable not $\sigma$-compact space also shows that the assumption of local compactness is essential in Theorem 2.12.

\begin{example} Let $X=\mathbb{R}\setminus\mathbb{Q}$ and let $\tau=\tau_{nat}\vert X$. Then the space of irrationals $(X, \tau)$ is second countable. That $(X, \tau)$ is not $\sigma$-compact in $\mathbf{ZF}$ can be shown by using the facts that the Baire category theorem holds in $\mathbf{ZF}$ in the class of separable completely metrizable spaces (see Theorem 4.102 of \cite{Her}) and that every compact subspace of $(X, \tau)$ is nowhere dense in $(X, \tau)$.  This is why the compact complement topology $(\tau_{nat}\vert X)^{\star}$ is not first countable, so it is not second countable. Of course, the space of irractionals is not locally compact at each one of its points. 
\end{example}

\begin{remark} It was shown in Theorem 2.7 of \cite{W} that if $\mathcal{T}$ is the co-finite topology on a set $Z$, then the space $(Z, \mathcal{T})$ is quasi-metrizable if and only if $Z$ is a countable union of finite sets. Now, suppose that $\tau$ is the discrete topology on $X$, i.e. $\tau$  is the power set $\mathcal{P}(X)$ of $X$. Then $\tau^{\star}$ is the co-finite topology on $X$. Hence, for $\tau=\mathcal{P}(X)$, the space $(X, \tau^{\star})$ is quasi-metrizable if and only if $X$ is a countable union of finite sets. In some models of $\mathbf{ZF}$ in which a countable union of finite sets can fail to be countable, even when $X$ is uncountable and $\tau=\mathcal{P}(X)$, then $(X, \tau^{\star})$ can be quasi-metrizable (see \cite{W}).
\end{remark}

The following question does not seem to be trivial:

\begin{question}  What are, expressed in terms of $\tau$,  simultaneously necessary and sufficient conditions for $(X, \tau^{\star})$ to be quasi-metrizable when $(X, \tau)$ is a $\sigma$-compact quasi-metrizable space?
\end{question}

\begin{remark} Let us consider the case when $(X, \tau)$ is not compact. We notice that if $p$ and $\hat{p}$ are properties such that a topological space $Z$ has $p$ if and only if $Z$ is Hausdorff and has $\hat{p}$, then, in view of Proposition 2.9, the space $(X, \tau^{\star})$ does not have $p$. In particular, $(X, \tau^{\star})$ is not a $T_i$-space for $i\in\{2, 3, 3\frac{1}{2}, 4, 5, 6\}$. It is easily seen that $(X, \tau^{\star})$ is neither regular, nor completely regular, nor normal. Every continuous mapping from $(X, \tau^{\star})$ to a Hausdorff space is constant.
\end{remark}

\begin{theorem} Let $A\subseteq X$. Then $A$ is $\tau^{\star}$-compact if and only if $A\cap K$ is $\tau$-closed for each $\tau$-compact set $K$.
\end{theorem}

\begin{proof} \emph{Necessity}. Suppose  that $A$ is $\tau^{\star}$-compact. Let $K$ be a $\tau$-compact set. Since $(X, \tau)$ is Hausdorff, to show that $A\cap K$ is $\tau$-closed, it suffices to check that $A\cap K$ is $\tau$-compact. Let $\mathcal{F}$ be a collection of $\tau$-closed sets such that the collection $\mathcal{H}=\{F\cap A\cap K: F\in\mathcal{F}\}$ is centred.  The sets $A\cap K\cap F$ for $F\in\mathcal{F}$ are all $\tau^{\star}\vert A$-closed. Since $A$ is $\tau^{\star}$-compact and $\mathcal{H}$ is a centred collection of $\tau^{\star}\vert A$-closed sets, we have  that $\bigcap\mathcal{H}\neq\emptyset$. This proves that $A\cap K$ is $\tau$-compact.

\emph{Sufficiency}. Now, suppose that $A\cap K$ is $\tau$-compact for each $\tau$-compact set $K$. We may assume that $A\neq\emptyset$.  Let $\mathcal{U}$ be a non-empty collection of non-empty sets such that $\mathcal{U}\subseteq \tau^{\star}$, while $A\subseteq\bigcup\mathcal{U}$. Fix any  non-empty set $U_0\in\mathcal{U}$. The set $C_0=X\setminus U_0$ is $\tau$-compact, so $A\cap C_0$ is $\tau$-compact as a $\tau$-closed subset of a $\tau$-compact set. Notice that $A\cap C_0\subseteq \bigcup\mathcal{U}$ and, by Theorem 2.2, $\mathcal{U}\subseteq\tau$. By the $\tau$-compactness of $A\cap C_0$, there exists a finite collection $\mathcal{V}\subseteq\mathcal{U}$ such that $A\cap C_0\subseteq\bigcup\mathcal{V}$. Then $A\subseteq U_0\cup\bigcup\mathcal{V}$. This proves that $A$ is $\tau^{\star}$-compact.
\end{proof}

\begin{corollary} For every Hausdorff space $(X, \tau)$, the space $(X, \tau^{\star})$ is compact.
\end{corollary}

A topological space $(Z, \mathcal{T})$ is called \emph{jointly partially metrizable on compact subspaces}, if there is a metric $d$ on $Z$ such that, for every compact subspace $A$ of $(Z, \mathcal{T})$, the restriction of $d$ to $A\times A$ generates the subspace topology $\mathcal{T}\vert A$ on $A$ (see  \cite{ArhShu}).

\begin{example}
The space $(\mathbb{R},\tau_{nat})$ is metrizable, hence jointly partially metrizable on compact subspaces. But $(\mathbb{R}, \tau_{nat}^{\star})$ is not jointly partially metrizable on compact subspaces since it is compact and not metrizable for it is not Hausdorff.
\end{example}

A topological space $Z$ is called C-\emph{normal} if there exists a normal space $Y$ and a bijective function $f : Z \rightarrow Y$ such that the restriction $f|A : A \rightarrow f(A)$ is a homeomorphism for each compact subspace $A$ of $Z$ (see \cite{KaZah}).

\begin{example}
The space $(\mathbb{R},\tau_{nat})$ is C-normal. But $(\mathbb{R}, \tau_{nat}^{\star})$ is not C-normal since it is compact and not normal.
\end{example}

\section{$k$-spaces}

Let us recall that a Hausdorff space $Z$ is called a \emph{k-space} if, for every set $A\subseteq Z$, it holds true that  $A$ is closed in $Z$ if and only if $A\cap K$ is closed in $Z$ for each compact set $K$ in $Z$ (see Section 3.3 of \cite{E}).

We deduce directly from Theorem 2.28  the following characterization of $k$-spaces:

\begin{theorem} For every Hausdorff space $(X, \tau)$, it holds true in $\mathbf{ZF}$ that $(X, \tau)$ is a $k$-space if and only if every $\tau^{\star}$-compact subset of $X$ is $\tau$-closed.
\end{theorem}

We recall definitions of sequential and Fr\'echet-Urysohn spaces for completeness.

\begin{definition} Let $Z$ be a topological space and $A\subseteq Z$. Then:
\begin{enumerate}
\item[(i)] $A^{s}$  denotes the set of all points $z\in Z$ such that there exists a
sequence $(z_{n})_{n\in\omega}$ of points of $A\setminus \{z\}$ which converges in $Z$ to the point $z$;
\item[(ii)] $A$ is called \emph{sequentially
closed} if $A^{s}\subseteq A$;
\item[(iii)]  the \emph{sequential closure} of $A$ in $Z$ is the set $\text{scl}_Z(A)=A^{s}\cup A$;
\item[(iv)] $Z$ is called \emph{sequential} (resp. \emph{Fr\'echet-Urysohn}) if every
sequentially closed subset of $Z$ is closed in $Z$ (resp. for every $F\in
\mathcal{P}(Z)$ the equality $\text{scl}_Z(F)=\text{cl}_Z(F)$ holds).
\end{enumerate}
\end{definition}

In some texts,  Fr\'echet-Urysohn spaces are called Fr\'echet spaces (see, for instance, \cite{E} and \cite{Her}). It is well known that the following series of implications hold true in $\mathbf{ZFC}$ and, in general, none of them
is reversible in $\mathbf{ZFC}$ (see, e.g. Sections 1.6 and Theorem 3.3.20 of \cite{E}):

\begin{itemize}
\item $Z$ is Hausdorff and first countable $\rightarrow $ $Z$ is Hausdorff and Fr\'echet-Urysohn $\rightarrow $
$Z$ is Hausdorff and sequential $\rightarrow $ $Z$ is a $k$-space.
\end{itemize}

Of course, the proof to Theorem 3.3.20 of \cite{E} shows that it is true in $\mathbf{ZF}$ that every Hausdorff sequential space is a $k$-space. That even $\mathbb{R}$ can fail to be sequential in a model of $\mathbf{ZF}$ is shown in Theorem 4.55 of \cite{Her}. The second part of Theorem 3.3.20 of \cite{E}, which states that every first countable Hausdorff space is a $k$-space, does not have a proof in $\mathbf{ZF}$. Therefore, since we work in $\mathbf{ZF}$, it is natural to ask about set-theoretical status of the following sentences:
\begin{enumerate}
\item[(a)] Every first countable Hausdorff space is a $k$-space.
\item[(b)] $\mathbb{R}$ is a $k$-space.
\item[(c)] Every subspace of $\mathbb{R}$ is a $k$-space.
\end{enumerate}

In this section,  we are going to prove that (a) is equivalent with the axiom of countable multiple choice (i.e. Form 126 in \cite{HR}), while (b) holds in $\mathbf{ZF}$ and (c) is independent of $\mathbf{ZF}$. We shall also show that even the Sorgenfrey line can fail to be a $k$-space in a model of $\mathbf{ZF}$.

We recall that the axiom of countable multiple choice, denoted by $\mathbf{CMC}$ in \cite{Her}, states that, for every  collection $\{A_n: n\in\omega\}$ of non-empty sets there exists a collection $\{F_n: n\in\omega\}$ of non-empty finite sets such that $F_n\subseteq A_n$ for each $n\in\omega$. It was shown in \cite{Ker1} that $\mathbf{CMC}$ is equivalent with Form 126D of \cite{HR}, i.e with the following sentence denoted by $\mathbf{WCMC}$:

$\textbf{WCMC}$: For every denumerable family $\mathcal{A}$ of disjoint non-empty sets there is an infinite set $C\subseteq \bigcup\mathcal{A}$ such that, for each $A\in\mathcal{A}$ the intersection $A\cap C$ is finite.

More information about $\mathbf{WCMC}$ can be found in \cite{Ker1} and in  Note 132 of \cite{HR}.

If $\mathcal{A}$ is a denumerable collection of pairwise disjoint non-empty sets, then every infinite set $C\subseteq\bigcup\mathcal{A}$ such that $C\cap A$ is finite for each $A\in\mathcal{A}$ is called a \emph{partial multiple choice} set of $\mathcal{A}$.

\begin{theorem} The following conditions are all equivalent in $\mathbf{ZF}$:
\begin{enumerate}
\item[(i)] $\mathbf{CMC}$;
\item[(ii)] every Hausdorff first countable space is a $k$-space;
\item[(iii)] every metrizable space is a $k$-space.
\end{enumerate}
\end{theorem}

\begin{proof}  Let $Y$ be a first countable Hausdorff space and let $D$ be a subset of $Y$ which is not closed in $Y$. Fix in $Y$ an accumulation point $y$ of $D$ such that $y\notin D$. Let $\mathcal{B}(y)=\{ U_n: n\in\omega\}$ be a countable base of open neighbourhoods of $y$ in $Y$ such that $U_{n+1}\subset U_n$ for each $n\in\omega$. Since $Y$ is Hausdorff, we can find a strictly increasing sequence $(k_n)_{n\in\omega}$ of positive integers such that the set $D_n= D\cap(U_{k_n}\setminus U_{k_{n+1}})$ is non-empty for each $n\in\omega$. Suppose  that $\mathbf{CMC}$ holds. By $\mathbf{CMC}$,  there exists a sequence $(C_n)_{n\in\omega}$ of non-empty finite sets such that $C_n\subseteq D_n$ for each $n\in\omega$. Then the set $C=\{y\}\cup\bigcup_{n\in\omega} C_n$ is compact in $Y$, while $y$ is an accumulation point of $D\cap C$ and $y\notin D\cap C$. Thus $D\cap C$ is not closed in $Y$. Therefore,  $Y$ is a $k$-space if $\mathbf{CMC}$ holds. Hence, (i) implies (ii). It is obvious that (ii) implies (iii). To complete the proof, it suffices to show that (iii) implies $\mathbf{WCMC}$.

Now, let us assume that $\mathbf{WCMC}$ is false. Suppose that $\mathcal{A}=\{A_n : n\in\omega\}$ is a collection of pairwise
disjoint non-empty sets without a partial multiple choice set. Put $A=\bigcup_{n\in\omega}A_n$. Take a point $\infty\notin A$ and put $Z=A\cup\{\infty\}$. Consider the metric $d$ on $Z$ defined in Lemma 2.22, as well as the topology $\tau(d)$ on $Z$ induced by $d$. Let $K$ be a compact subspace of $(Z, \tau(d))$. Since each $A_n$ is a discrete clopen subspace of $(Z, \tau(d))$, the sets $K\cap A_n$ are all finite. If $K$ were infinite, then $K$ would be a multiple choice set of $\mathcal{A}$. Hence, $K$ is finite, so $A\cap K$ is compact in $(Z, \tau(d))$. By Lemma 2.22,  $A$ is not closed in $(Z, \tau(d))$. This shows that $(Z, \tau(d))$ is not a $k$-space. Hence, (iii) implies (i).
\end{proof}

\begin{corollary} It is consistent with $\mathbf{ZF}$ that not every metrizable space is a $k$-space.
\end{corollary}

\begin{theorem} If $\mathcal{M}$ is a model of $\mathbf{ZF}$ in which every metrizable space is sequential, then $\mathbf{CMC}$ holds in $\mathcal{M}$.
\end{theorem}
\begin{proof} Suppose $(Z, \tau(d))$ is the space from Lemma 2.22 and the proof to Theorem 3.3 where $\mathcal{A}=\{A_n : n\in\omega\}$ is a collection of pairwise disjoint non-empty sets without a partial multiple choice set. Then the set $A$ is sequentially closed but not closed in $(Z, \tau(d))$.
\end{proof}

\begin{remark} Let us notice that since $\mathbf{CMC}$ implies $\mathbf{CC}(\mathbb{R})$, it follows directly  from  Exercise E.3 to Section 4.6  of \cite{Her} that in every  model of $\mathbf{ZF}$ in which $\mathbf{CMC}$ holds, every second countable $T_ 0$-space (in particular, every second countable metrizable space) is Fr\'echet-Urysohn, so sequential.
\end{remark}

\begin{theorem} $\mathbb{R}$ is a $k$-space in every model of $\mathbf{ZF}$.
\end{theorem}
\begin{proof} Let $A$ be a  subset  of $\mathbb{R}$ such that $A\cap K$ is closed in $\mathbb{R}$ for each compact set $K$ in $\mathbb{R}$. Suppose that  $x\in (\text{cl}_{\mathbb{R}}A)\setminus A$. Let $K_n=A\cap [x-\frac{1}{2^n}, x+\frac{1}{2^n}]$ for each $n\in\omega$. The sets $K_n$ are all non-empty and compact in $\mathbb{R}$. We put $x_n=\inf( K_n)$ for each $n\in\omega$. It follows from the compactness of $K_n$ that $x_n\in K_n$ for each $n\in\omega$. In this way, we define a sequence $(x_n)_{n\in\omega}$ of points of $A$ which converges in $\mathbb{R}$ to $x$. The set $K=\{x\}\cup\{ x_n: n\in\omega\}$ is compact in $\mathbb{R}$ but $A\cap K$ is not closed in $\mathbb{R}$ which is a contradiction. Hence, $A$ must be closed in $\mathbb{R}$. This implies that $\mathbb{R}$ is a $k$-space in $\mathbf{ZF}$.
\end{proof}

\begin{proposition} (i) It is consistent with $\mathbf{ZF}$ that a subspace of $\mathbb{R}$ can fail to be a $k$-space.\\
(ii) It is consistent with $\mathbf{ZF}$ that all subspaces of $\mathbb{R}$ are $k$-spaces.
\end{proposition}
\begin{proof} (i) Suppose that $X$ is an infinite Dedekind finite subset of $\mathbb{R}$. Since $X$ as a subspace of $\mathbb{R}$  is not discrete, there exists a set $A\subseteq X$ such that $A$ is not closed in $X$. Let $K$ be a compact subset of $X$. Then $K$ is compact in $\mathbb{R}$, so, if $K$ were infinite, then $K$ would be Dedekind infinite. Since $K$ is Dedekind finite, we deduce that $K$ is finite. This implies  $A\cap K$ is closed in $X$ because $A\cap K$ is finite. To complete the proof to (i), it  suffices to notice that in the model $\mathcal{M}1$ of \cite{HR} there is an infinite Dedekind finite subset of $\mathbb{R}$.

(ii) Let $\mathcal{M}$ be a model of $\mathbf{ZF}$ in which $\mathbf{CC}(\mathbb{R})$ holds. For instance, the model $\mathcal{M}2$ of \cite{HR} can be taken as $\mathcal{M}$. Since, by Theorem 4.54 of \cite{Her}, it is true in $\mathcal{M}$ that every subspace of $\mathbb{R}$  is sequential, we infer that, in $\mathcal{M}$, every subspace of $\mathbb{R}$ is a $k$-space.
\end{proof}

\begin{corollary} It is independent of $\mathbf{ZF}$ that all subspaces of $\mathbb{R}$ are $k$-spaces.
\end{corollary}

In what follows, as a metric space, $\mathbb{R}$ is considered with the metric $\rho$ defined by $\rho(x, y)=\vert x-y\vert$ for all $x, y\in\mathbb{R}$.

Using the notation from Theorem 4.55 of \cite{Her}, we denote by $\mathbf{CC}(c\mathbb{R})$ the following sentence: Every non-empty countable collection of non-empty complete subspaces of $\mathbb{R}$ has a choice function.

\begin{theorem}  (i) If the Sorgenfrey line is a $k$-space, then $\mathbf{CC}(c\mathbb{R})$ holds.\\
(ii) If $\mathbf{CC}(\mathbb{R})$ holds, then the Sorgenfrey line is a $k$-space.
\end{theorem}
\begin{proof} (i) Suppose that $\mathbf{CC}(c\mathbb{R})$ does not hold. Then, by Theorem 4.55 of \cite{Her}, $\mathbb{R}$ is not sequential. Let $A$ be a sequentially closed subset  of $\mathbb{R}$ which is not closed in $\mathbb{R}$. Let $a\in(\text{cl}_{\mathbb{R}}A)\setminus A$. The set $B=[(A-a)\cup (-A+a)]\cap (0, +\infty)$ is sequentially closed in $\mathbb{R}$ and not closed in $\mathbb{R}$. Since $0\in(\text{cl}_{\mathbb{S}}B)\setminus B$, the set $B$ is not closed in $\mathbb{S}$. Let $K$ be a compact set in $\mathbb{S}$. Then $K$ is countable and compact in $\mathbb{R}$. The set $K\cap B$ is sequentially closed in $\mathbb{R}$ and since, in addition, $K\cap B$ is countable, we deduce that $K\cap B$ is closed in $\mathbb{R}$. This implies $K\cap B$ is closed in $\mathbb{S}$. Therefore, $\mathbb{S}$ is not a $k$-space.

(ii) Now, suppose that $\mathbf{CC}(\mathbb{R})$ holds. Let $F\subseteq \mathbb{R}$ be not closed in $\mathbb{S}$ and let $x\in\text{cl}_{\mathbb{S}}F\setminus F$. Then $G=F\cap (x, +\infty)$ is not closed in  $\mathbb{R}$ and $x\in(\text{cl}_{\mathbb{R}}G)\setminus G$. In the light of Theorem 4.54 of \cite{Her}, $\mathbb{R}$ is Fr\'echet. This implies that there exists a sequence $(x_n)_{n\in\omega}$ of points of $G$ which converges in $\mathbb{R}$ to $x$. The set $K=\{x\}\cup\{x_n: n\in\omega\}$ is compact in $\mathbb{S}$ but $K\cap F$ is not closed in $\mathbb{S}$. This proves that $\mathbb{S}$ is a $k$-space.
\end{proof}

\begin{corollary} It  is consistent with $\mathbf{ZF}$ that the Sorgenfrey line is not a $k$-space.
\end{corollary}

From Theorems 2.28  and 3.10, we immediately obtain the following:

\begin{corollary} Let $\tau$ be the topology of the Sorgenfrey line. If every compact in $(\mathbb{R}, \tau^{\star})$ set is closed in $(\mathbb{R}, \tau)$, then $\mathbf{CC}(c\mathbb{R})$ holds.
\end{corollary}

\section{Compact complement partial topology}

Let us slightly reformulate Definition 2.1 of \cite{OPSW}:

\begin{definition} A \emph{partial topology} of a set $X$ is a collection $Cov_X\subseteq\mathcal{P}(\mathcal{P}(X))$ which satisfies the following conditions:
\begin{enumerate}
\item[(i)] $\tau_X=\bigcup Cov_X$ is a topology on $X$;
\item[(ii)] if $\mathcal{U}\subseteq \tau_X$ and $\mathcal{U}$ is finite, then $\mathcal{U}\in Cov_X$;
\item[(iii)] if $\mathcal{U}\in Cov_X$ and $V\in\tau_X$, then $\{U\cap V: U\in\mathcal{U}\}\in Cov_X$;
\item[(iv)] if $\mathcal{U}\in Cov_X$ and, for each $U\in\mathcal{U}$, we have $\mathcal{V}(U)\in Cov_X$ such that $U=\bigcup\mathcal{V}(U)$, then $\bigcup_{U\in\mathcal{U}}\mathcal{V}(U)\in Cov_X$;
\item[(v)] if $\mathcal{U}\subseteq \tau_X$ and $\mathcal{V}\in Cov_X$ are such that $\bigcup\mathcal{U}=\bigcup\mathcal{V}$ and, for each $V\in\mathcal{V}$, there exists $U\in\mathcal{U}$ such that $V\subseteq U$, then $\mathcal{U}\in Cov_X$.
\end{enumerate}
\end{definition}

\begin{definition} If  $Cov_X$ is a partial topology on a set $X$, then the ordered pair $(X, Cov_X)$ is called a \emph{partially topological space}, while $\tau_X=\bigcup Cov_X$ is called the \emph{topology corresponding to} $Cov_X$. 
\end{definition}

 \begin{remark} Let us notice that if $(X, Cov_X)$ is a partially topological space, then the triple $(X,\bigcup Cov_X, Cov_X)$ is a Delfs-Knebusch generalized topological space (in abbreviation a D-K gts) in the sense of Definition 2.2.1 of \cite{Pie1} and, moreover, this D-K gts is partially topological in the sense of Definition 2.2.67 of \cite{Pie1}.  Delfs-Knebusch gtses were studied, for instance, in \cite{DK,K,OPSW,Pie1,Pie2,Pie3,PW, PW1}. We recall that, according to Remark 2.2.3 of \cite{Pie1}, a D-K gts is an ordered pair $(X, Cov_X)$ such that, for $Op_X=\bigcup Cov_X$, the triple $(X,Op_X, Cov_X)$ satisfies the conditions of Definition 2.2.2 of \cite{Pie1}. In general, $Op_X$ need not be a topology on $X$. If $(X, Cov_X)$ is a D-K gts, then $Cov_X$ is called a D-K (Delfs-Knebusch) generalized topology on $X$. 
\end{remark}

If $\psi$ is a topological property, then we say that a partially topological space $(X, Cov_X)$ has $\psi$ if the topological space $(X, \bigcup Cov_X)$ has $\psi$. In particular:

\begin{definition} We say that a partially topological space $(X. Cov_X)$ is:
\begin{enumerate}
\item[(i)] \emph{Hausdorff} if $(X, \bigcup Cov_X)$ is Hausdorff;
\item[(ii)]\emph{compact} if $(X, \bigcup Cov_X)$  is compact.
\end{enumerate}
\end{definition}

\begin{definition}
Let  $(X, Cov_X)$ be a Hausdorff partially topological space,  $\tau_X$  the topology corresponding to $Cov_X$ and $\tau_X^{\star}$ the compact complement  topology of $(X, \tau_X)$. Then the collection
$$ Cov_X^{\star}=Cov_X \cap \mathcal{P}(\tau_X^{\star})$$
will be called the \emph{compact complement partial topology} of $(X, Cov_X)$.
\end{definition}

\begin{remark}
Let  $(X, Cov_X)$ be a Hausdorff partially topological space.  That $Cov_X^{\star}$ is a D-K generalized topology follows from Fact 2.2.31 in \cite{Pie1} which says that  the intersection of any non-empty family of D-K generalized topologies on $X$ is a D-K generalized topology on $X$. Since $\bigcup (Cov_X \cap \mathcal{P}(\tau_X^{\star}))=\tau_X^{\star}$, the D-K generalized topology  $Cov_X^{\star}$ is a partial topology on $X$.
\end{remark}

In what follows, we use the  symbols $\cap_1, \setminus_1$ introduced on page 219 of \cite{Pie1}. We recall that, for  collections $\mathcal{U}, \mathcal{V}$ of subsets of $X$, we have $\mathcal{U}\cap_1\mathcal{V}=\{ U\cap V: U\in\mathcal{U}, V\in\mathcal{V}\}$ and, analogously,  $\mathcal{U}\setminus_1\mathcal{V}=\{U\setminus V: U\in\mathcal{U}, V\in\mathcal{V}\}$. Moreover, for a collection $\mathcal{A}\subseteq\mathcal{P}(\mathcal{P}(X))$, we denote by $\langle \mathcal{A}\rangle_X$ the intersection of all D-K generalized topologies on $X$ that contain $\mathcal{A}$ (see page 242 of \cite{PW}).

\begin{definition}
 For a subset $Y$ of $X$ and a partial topology $Cov$ on $X$, let
$$Cov\vert Y=\langle \{\mathcal{V}\cap_1 \{ Y\} : \mathcal{V}\in  Cov\} \rangle_Y.$$
Then $Cov\vert Y$ is called the \emph{partial topology on $Y$ induced by} $Cov$ and
$(Y, Cov\vert Y)$ is called  a partially topological \emph{subspace} of
$(X, Cov)$.
\end{definition}

The following  theorem is an adaptation of Theorem 2.2 to partial topologies:

\begin{theorem} Let $(X, Cov_X)$ be a Hausdorff partially topological space, $\tau_X=\bigcup Cov_X$ and
 $Y\subseteq X$. The following conditions are fuilfilled:
\begin{enumerate}
\item[(i)] $Cov_X^{\star}\vert Y\subseteq Cov_X\vert Y$;
\item[(ii)] if $Y$ is compact in $(X, \tau_X)$, then $Cov_X^{\star}\vert Y=Cov_X\vert Y$;
\item[(iii)] $Cov_X^{\star}\vert Y=Cov_X\vert Y$ if and only if there exists a $\tau_X$-compact set $K$ such that $Y\subseteq K$.
\end{enumerate}
\end{theorem}

\begin{proof}  Since  $Cov_X^{\star}\subseteq Cov_X$, it is obvious that (i) is satisfied.

(ii)  Suppose that $Y$ is $\tau_X$-compact and $\mathcal{V}\in Cov_X$. Since $(X, \tau_X)$ is Hausdorff, the set $Y$ is $\tau_X$-closed, so $\mathcal{A}=\{Y\}\cap_1 (\{X\}\setminus_1 \mathcal{V})$ is a collection of  $\tau_X$-compact subsets of $Y$.
Notice that $\mathcal{V}\cap_1 \{Y\}=\{Y\}\cap_1 (\{X\}\setminus_1 \mathcal{A})$. This implies that $\mathcal{V}\cap_1 \{Y\}\in Cov_X^{\star}\vert Y$ and, in consequence,  $Cov_X\vert Y\subseteq  Cov_X^{\star}\vert Y$.

(iii) Now, assume that $K$ is a $\tau_X$-compact set such that $Y\subseteq K$. It follows from (ii) that $Cov_X\vert K=Cov_X^{\star}\vert K$. Hence, in view of Fact 10.3 of \cite{PW}, we have $Cov_X\vert Y= (Cov_X\vert K)\vert Y=(Cov_X^{\star}\vert K)\vert Y= Cov_X^{\star}\vert Y$.

Finally,  suppose that $Y$ is a subset of $X$ such that $Cov_X^{\star}\vert Y=Cov_X\vert Y$. Let $V\in \tau_X$ be such that $\emptyset\neq V \cap Y\neq Y$. Then $\{V\}\in Cov_X$.  Since $\{V\cap Y\}\in Cov_X^{\star}\vert Y$, there exists a $\tau_X$-compact  set $K_0$ such that $V \cap Y=Y\setminus {K}_0$. Reasoning as in the proof to Theorem \ref{theorem1} (iii), we get that there exists a $\tau_X$-compact set $K$ such that  $Y\subseteq K$.
\end{proof}

\begin{corollary} A Hausdorff partially topological space $(X, Cov_X)$ is compact if and only if
$Cov_X=Cov_X^{\star}$.
\end{corollary}

In view of Proposition 2.8 and Corollary 2.29,  the following proposition holds:

\begin{proposition} If $(X, Cov_X)$ is a Hausdorff partially topological space,  then the partially topological space  $(X, Cov_X^{\star})$ is compact and $T_1$.
\end{proposition}

\begin{remark}
Similarly to the situation in  Remark 2.4,  we have that, in general,  $Cov_X^{\star}\vert Y$ need not be equal to $(Cov_X \vert Y)^{\star}$.
\end{remark}

Although it can be said more about compact complement partial topologies, let us finish with the following example:

\begin{example}\label{example2-1}
Consider the partially topological real lines considered in
Definition 1.2 of  \cite{PW}: $\mathbb{R}_{st}=(\mathbb{R},Cov_{st}), \mathbb{R}_{lst}=(\mathbb{R},Cov_{lst}), \mathbb{R}_{l^+st}=(\mathbb{R},Cov_{l^+st})$. Let $I$ be a  bounded interval of $\mathbb{R}$.
Then we get the following  equalities of the induced partial topologies:
 $Cov_{st}\vert I=Cov_{lst}\vert I=Cov_{l^+st}\vert I=
Cov_{st}^{\star}\vert I=Cov_{lst}^{\star}\vert I=
Cov_{l^+st}^{\star}\vert I$.
\end{example}


\end{document}